\documentclass{amsart}
\usepackage{amsfonts,amscd,amssymb,amsmath,amsthm}
\usepackage{graphicx,mathrsfs,appendix}
\usepackage{centernot,caption,subcaption,color,url}
\usepackage[T1]{fontenc}
\usepackage{ctable}

\begin{document}
\newtheorem{theorem}{Theorem}
\newtheorem{lemma}[theorem]{Lemma}
\newtheorem{conjecture}[theorem]{Conjecture}
\newtheorem{cor}[theorem]{Corollary}
\newtheorem{proposition}[theorem]{Proposition}
\newtheorem{claim}[theorem]{Claim}
\newtheorem{remark}[theorem]{Remark}
\renewcommand{\P}{\mathbb{P}}
\newcommand{\R}{\mathbb{R}}
\newcommand{\T}{\mathcal{T}}
\newcommand{\B}{\mathcal{B}}
\newcommand{\A}{\mathcal{A}}
\newcommand{\G}{\mathcal{G}}
\newcommand{\F}{\mathcal{F}}
\newcommand{\Z}{\mathbb{Z}}
\newcommand{\Q}{\mathbb{Q}}
\newcommand{\E}{\mathbb E}
\newcommand{\N}{\mathbb N}
\newcommand{\1}{\mathbf 1}

\title[Preferential attachment with choice based edge-step]{Preferential attachment with choice based edge-step}
\author{Yury Malyshkin}
\address{Tver State University}
\email{yury.malyshkin@mail.ru}
\date{\today}

\begin{abstract}
We study the asymptotic behavior of the maximum degree in the preferential attachment model with a choice-based edge-step. We add vertex type to the model and prove, among others types of behavior, the effect of condensation on multiple vertices with different types.
\end{abstract}


\maketitle

\section{Introduction}

In the present work, we study the addition of vertex fitness to the linear preferential attachment model (see, e.g. \cite{Mori02,Mori05}) with a choice-based edge step. The addition of the vertices of different types is a natural way to model people's preferences, for example in electoral models (see, e.g., models with fitness \cite{BCDR07} or geometrical models \cite{HJY20}). 

The standard preferential attachment graph model was introduced in \cite{BA99}.  Such a graph is constructed in the following way. First, we start with some initial graph $G_1$, usually, for simplification purposes, it consists of two vertices and an edge between them. Then on each step, we add a new vertex and draw an edge from it to an existing vertex, chosen with probability proportional to its degree. Usually, one considers the rule when we choose a vertex with probability proportional to its degree plus some parameter $\beta>-1$. Such a model was widely studied (see, e.g., \cite{Hof16}, section 8) and different modifications have been introduced.

One of the modifications we consider is the introduction of a choice to the model (see, e.g., \cite{HJ16,KR14,Mal18,MP14}). In this modification, we consider the sample of $d$ independently chosen vertices and then choose the one with the largest degree. This modification often results in the effect of condensation, when a single vertex has linear (over the total number of edges) degree (see, e.g., \cite{HJY20,Mal20,MP15}). The other modification is the introduction of the edge step to the model (see, e.g., \cite{ARS21}). In this modification, we have two types of steps. The vertex step is the classical preferential attachment step when we add a new vertex and then draw edges from it to an already existing vertex. The other type of step is an edge step. In this step, we draw edges between already existing vertices, chosen with probabilities proportional to their degrees.

Let us introduce our model. Fix $k,m,d,T\in\N$, $d>1,$ $\beta>-1$ and $p_i\in(0,1),$ $i=1,...,T$, $\sum_{i=1}^{T}p_i=1,$ which are parameters of our model. We consider a sequence of graphs $G_n$, with $G_n$ containing vertices $v_1,...,v_n$. We consider i.i.d. random variables $X_1,X_2,...$ with distributions $P(X_1=i)=p_i,$  such that $X_i$ corresponds to $v_i$ and represents the type (fitness) of a vertex. We start with the initial graph $G_1$ that consists of a vertex $v_1$.  To build graph $G_{n+1}$ from $G_n$ we add a vertex $v_{n+1}$ and draw edges in two steps.
\begin{itemize}
\item[Vertex step:] we draw $m$ edges independently from $v_{n+1}$ to one of the vertices $v_1,v_2,...,v_n$, for each edge, the endpoint is chosen with conditional probability (given $G_n$)
\begin{equation}
    \frac{\deg_{G_n} (v_i) + \beta}{\beta n + \sum_{i=1}^{n} \deg_{G_n} (v_i)}.
\end{equation}
\item[Edge step:] we independently choose $k$ pairs $(w_n^{i},u_n^{i})$, $i=1,...,k$, of vertices of $G_{n}\cup v_{n+1}$, such that the vertex $w_n^{i}$ is chosen uniformly from all vertices and the vertex $u_n^{i}$ is chosen by the following rule. We consider a sample $y_n^{i,1},...,y_n^{i,d}$ of size $d$ of vertices of $G_{n}\cup v_{n+1}-w_n^{i}$ with the same type $t$ as $w_n^{i}$ (or with any type if there are no other vertices of type $t$), chosen with conditional (given $G_n$) probabilities
\begin{equation}
\Pr(y_n^{i,k}=v_j|\F_n)=\1\{X_j=t\}\frac{\deg_{G_n} (v_j) + \beta}{\sum_{s=1}^{n} \1\{X_s=t\}\left(\deg_{G_n} (v_s)+\beta\right)}.
\end{equation}
Here we use $\F_n$ to describe the $\sigma$-algebra generated by all the random variables used in constructing $G_n$. 
Then we draw an edge from $w_n^{i}$ to the vertex from the sample with the highest degree in $G_n$ (in case of a tie chosen randomly, it would not affect the degree distribution).
\end{itemize} 

\section{Results}

Let $D_{i}(n) = \sum_{j=1}^{n}\left(\deg_{G_n}(v_j)+\beta\right)\1\{X_j =i\}$ be the total weight of all vertices of type $i$ and $D(n)=\sum_{i=1}^{T}D_{i}(n)$. Note that $D(n)$ is the sum of weights of all vertices of the graph, so $D(n)=(2m + 2k + \beta)n$. At step $n+1$, we could increase $D_{i}(n)$ in the following ways. We could draw one of $m$ edges to the vertex of $G_n$ of type $i$ during a vertex step, with the expected number of such edges equal to $m\frac{D_{i}(n)}{D(n)}+o(e^{-cn})$. The new vertex could also be of type $i$, which happens with probability $p_i$ and results in the increase of $D_{i}(n)$ by $m+\beta$. We also could increase $D_{i}(n)$ during an edge step (if there are at least two vertices of type $i$, which happens with probability $1+o(e^{-cn})$ for some $c>0$) with expected increase (conditional on $\F_n$) equals to $2kp_in+o(n^{1/2}\ln n)$. Therefore, we get representation
$$\E\left(D_{i}(n+1)-D_{i}(n)|\F_n\right)=m\frac{D_{i}(n)}{D(n)}+(m + 2k + \beta)p_{i}n+o(n^{1/2}\ln n).$$
Hence, due to convergence theorems for stochastic approximation (see, e.g., \cite[Theorem~3.1.1]{Chen03}), 
\begin{equation}
\label{eq:D_estimate}
    D_{i}(n) = (2m + 2k + \beta)p_{i}n + o(n^{2/3}) \text{ almost surely.} 
\end{equation}

Let us formulate our main result. Let $M_1(n),...,M_T(n)$ be the highest degrees of vertices of types $1,...,T$ in $G_n$. Let define function
\begin{equation}
    \label{eq:func}
f(x):=\frac{m}{2m+2k+\beta}x+k\left(1-\left(1-\frac{x}{2m+2k+\beta}\right)^d\right)
\end{equation}
\begin{theorem}
\label{th:main}
 \begin{enumerate}
\item If $\frac{m+dk}{2m+2k+\beta}<1$ than for any $\epsilon>0$ 
$$\Pr\left(\forall n>n_0: n^{ \frac{m+dk}{2m+2k+\beta}-\epsilon}<M_i(n)<n^{ \frac{m+dk}{2m+2k+\beta}+\epsilon}\right)\to 1$$
as $n_0\to\infty$.
\item If $\frac{m+dk}{2m+2k+\beta}=1$, than almost surely
$$\liminf_{n\to\infty}\frac{M_i(n)\ln n}{n}=\frac{2p_1(2m+2k+\beta)^2}{kd(d-1)}.$$
\item If $\frac{m+dk}{2m+2k+\beta}>1$, than almost surely
$$\liminf_{n\to\infty}\frac{M_i(n)}{n}=p_ix^{\ast},$$
where $x^{\ast}$ is a unique positive root of the equation $f(x)=x$.
\end{enumerate}
\end{theorem}
Note that the unique root exists due to the concavity of $f(x)$ and the condition $f^{\prime}(0) > 1$.
This theorem shows the existence of a condensation effect on multiple vertices.

Since types of vertices are i.i.d. random vertices with a finite number of values, the number $N_i(n)$ of vertices of type $i$ in $G_n$ satisfies the law of iterated logarithm. Hence, due to \eqref{eq:D_estimate},
for any $\epsilon>0$ 
$$\Pr\left(\A^{i}_{\epsilon}(n_0)\right)\to 1\quad \text{ as }\quad n_0\to\infty$$
where 
\begin{equation}
\A^{i}_{\epsilon}(n_0)=\left\{\forall n\geq n_0, \; D_{i}(n)<(2 m+2 k+\beta + \epsilon)p_{i}n \text{ and } |N_i(n)-p_in|<\epsilon n\right\}. 
\end{equation}

\section{Stochastic approximation and auxiliary processes}
To prove the above theorem, we need the following auxiliary result from stochastic approximation processes (see Corollary 2.7 in \cite{Pem07} and \cite{Chen03}  for more details).
\begin{lemma}
\label{lem:approx}
Let $\F_n$-measurable process $Z(n)$ satisfy the following conditions:
\begin{itemize}
\item $|Z(n+1)-Z(n)|<C$ almost surely for some constant $C$.
\item $\E(Z(n+1)-Z(n)|\F_n)= f\left(\frac{Z(n)}{n}\right)+O\left(\frac{1}{n}\right)$ for some concave function $f(x)$.
\item $f(0)=0$, $f^{\prime}(0)>1$, $f(c)=0$ for some $c>0$.
\end{itemize} 
Then, almost surely,
$$\lim_{n\to\infty}\frac{Z(n)}{n}=c.$$
\end{lemma} 
We now use it to study the process of the specific form.

\begin{lemma}
\label{lem:approximation}
Let $\F_n$-measurable process $Y(n)$ with values in $\N$ with non-negative bounded increments (i.e. $0\leq Y(n+1)-Y(n)\leq C$) satisfy
$$\E(Y(n+1)-Y(n)|\F_n)=g\left(\frac{Y(n)}{n}\right)+O\left(\frac{1}{n}\right),$$
where $g(x)$ satisfies $g(0)=0$, $g^{\prime}(0)>0$ and $g^{\prime\prime}(x)<0$.
\begin{enumerate}
\item If $g^{\prime}(0)<1$, then for any $\epsilon>0$ 
$$\Pr\left(\forall n>n_0: n^{g^{\prime}(0)-\epsilon}<Y(n)<n^{g^{\prime}(0)+\epsilon}\right)\to 1$$
as $n_0\to\infty$.
\item If $g^{\prime}(0)=1$, then, almost surely,
$$\liminf_{n\to\infty}\frac{Y(n)\ln n}{n}=\frac{2}{-g^{\prime\prime}(0)}.$$
\item If $g^{\prime}(0)>1$, then, almost surely,
$$\liminf_{n\to\infty}\frac{Y(n)}{n}=x^{\ast},$$
where $x^{\ast}$ is a unique positive root of the equation $g(x)=x$.
\end{enumerate}
\end{lemma}
\begin{proof}
First, we apply lemma~\ref{lem:approx}. To do so, introduce $Z(n)=\frac{Y(n)}{n}$. We get
$$\E(Z(n+1)-Z(n)|\F_n)=\frac{1}{n+1}\left(\E(Y(n+1)-Y(n)|\F_n)-Z(n)\right).$$
Hence,
$$\E(Z(n+1)-Z(n)|\F_n)=\frac{1}{n+1}\left(f(Z(n))+O\left(\frac{1}{n}\right)\right),$$
where $f(x):=g(x)-x$. Note that $f(x)$ is concave and $f^{\prime}(0)=g^{\prime(0)}-1$.
Due to results on stochastic approximation, $Z(n)$ converges almost surely to the non-negative zero set of $f(x)$. In case $g^{\prime}(0)>1$, there are two non-negative zeros: $0$ and $x^{\ast}>0$. In case $g^{\prime}(0)\leq 1$, $f(x)<0$ when $x>0$, so $Z(n)\to 0$ almost surely.

Let consider $W(n+1)=\frac{Y(n+1)}{Y(n)}$. Note that
\begin{equation}
\label{eq:upper}
\E\left(W(n+1)|\F_n\right)=1+\frac{g\left(\frac{Y(n)}{n}\right)}{Y(n)}+O\left(\frac{1}{n}\right). 
\end{equation}
We also have
$$\E\left(\frac{1}{W(n+1)}|\F_n\right)=\E\left(\frac{Y(n+1)-(Y(n+1)-Y(n))}{Y(n+1)}|\F_n\right)$$
$$=1-\E\left(\frac{Y(n+1)-Y(n)}{Y(n)}\frac{1}{1+\frac{Y(n+1)-Y(n)}{Y(n)}}|\F_n\right)$$
where $0\leq \frac{Y(n+1)-Y(n)}{Y(n)}\leq\frac{C}{Y(n)}$. Hence,
$$\E\left(\frac{1}{W(n+1)}|\F_n\right)=1-\frac{g\left(\frac{Y(n)}{n}\right)}{Y(n)}\left(1+O_{+}\left(\frac{1}{Y(n)}\right)\right)+O\left(\frac{1}{nY(n)}\right).$$
Using Taylor formula for $g(x)$, we get
\begin{equation}
\label{eq:lower}    \E\left(\frac{1}{W(n+1)}|\F_n\right)=1-\left(\frac{g^{\prime}(0)}{n}+\frac{g^{\prime\prime}(0)}{2n}\frac{Y(n)}{n}\right) \left(1+O_{+}\left(\frac{1}{Y(n)}\right)\right)+O\left(\left(\frac{Y^2(n)}{n^3}\right)\right).
\end{equation}
Let fix small $\epsilon>0$ and consider $A_{\epsilon}(n)=\frac{n^{g^{\prime}(0)-\epsilon}}{Y(n)}$. On this event, $\left\{\frac{Y(n)}{n}\to 0\right\}$ for large enough $n$ with high probability
$$\E\left(\frac{A_{\epsilon}(n+1)}{A_{\epsilon}(n)}|\F_n\right) =\left(1+\frac{1}{n}\right)^{g^{\prime}(0)-\epsilon}\E\left(\frac{1}{W(n+1)}|\F_n\right)$$
$$\leq \left(1+\frac{g^{\prime}(0)-\epsilon/2}{n}\right)\left(1-\frac{g^{\prime}(0)-\epsilon/2}{n}\right)<1.$$
Hence, on the event $\left\{\frac{Y(n)}{n}\to 0\right\}$, $A_{\epsilon}(n)$ is a supermartingale, and, by the martingale converges theorem, it converges to a random variable $R$. Therefore, $Y(n)>n^{g^{\prime}(0)-2\epsilon}$ with a high probability for large enough $n$ on $\left\{\frac{Y(n)}{n}\to 0\right\}$. In particular, that implies non-convergence of $Y(n)$ to $0$ in the case $g^{\prime}(0)>1$, which results in $Y(n)/n \to x^{\ast}$ almost surely. 

Now let us obtain the matching upper bound for the case $g^{\prime}(0)<1$. Let consider $B_{\epsilon}(n)=\frac{Y(n)}{n^{g^{\prime}(0)+\epsilon}}$.  From \eqref{eq:lower} we get (for small $\epsilon>0$)
$$\E\left(\frac{B_{\epsilon}(n+1)}{B_{\epsilon}(n)}|\F_n\right)$$
$$=\left(\frac{1}{1+\frac{1}{n}}\right)^{g^{\prime}(0)+\epsilon} \left(1+\frac{g\left(\frac{Y(n)}{n}\right)}{Y(n)}+O\left(\frac{1}{n}\right)\right)$$
$$=\left(\frac{1}{1+\frac{1}{n}}\right)^{g^{\prime}(0)+\epsilon} \left(1+\frac{g^{\prime}(0)+O\left(\frac{Y(n)}{n}\right)}{n}+O\left(\frac{1}{n}\right)\right)$$
$$=\left(1+\frac{-\epsilon+O\left(\frac{Y(n)}{n}\right)}{n}+O\left(\frac{1}{n}\right)\right)$$
$$< 1$$
for large enough $n$.
Therefore, $B_{\epsilon}(n)$ is a supermartingale, and, with high probability, $Y(n)<n^{g^{\prime}(0)+2\epsilon}$ for large enough $n$.

Finally, let us consider the case $g^{\prime}(0)=1$. Note that 
$$\E\left(Y(n+1)-Y(n)|\F_n\right)=\sum_{k}k\P\left(Y(n+1)-Y(n)=k|\F_n\right),$$
where the sum is over finite number of values that $Y(n+1)-Y(n)$ could take.
From \eqref{eq:lower} we get
Fix a constant $c>0$. Consider
$$U_c(n):=ne^{-\frac{cn}{Y(n+1)}}.$$
We get 
$$\E\left(\frac{U_c(n+1)}{U_c(n)}|\F_n\right)=\frac{n+1}{n}\E\left(e^{c\left(\frac{n}{Y(n)}-\frac{n+1}{Y(n+1)}\right)}|\F_n\right)$$
$$=\frac{n+1}{n}\E\left(e^{\frac{cn(Y(n+1)-Y(n))}{Y(n)Y(n+1)}-\frac{c}{Y(n+1)}}|\F_n\right)$$
$$=\frac{n+1}{n}\sum_{k}\left(1+e^{\frac{cnk}{Y(n)(Y(n)+k)}}-\frac{c}{Y(n)+k}\right)\P\left(Y(n+1)-Y(n)=k|\F_n\right).$$
Since for any $\epsilon>0$, $Y(n)$ grows faster than $n^{1-\epsilon}$ and $\frac{Y(n)}{n}\to 0$ almost surely, we get 
$$e^{\frac{cnk}{Y(n)(Y(n)+k)}}-\frac{c}{Y(n)+k}=1+\frac{cnk}{Y^{2}(n)}-\frac{c}{Y(n)}+o\left(\frac{1}{n}\right).$$
Also, $\sum_kk\P\left(Y(n+1)-Y(n)=k|\F_n\right)=g\left(\frac{Y(n)}{n}\right)$ and $\sum_k\P\left(Y(n+1)-Y(n)=k|\F_n\right)=1$. Using Taylor expansion of $g$, we get
$$\E\left(\frac{U_c(n+1)}{U_c(n)}|\F_n\right) =\frac{n+1}{n}\left(1+\frac{cng\left(\frac{Y(n)}{n}\right)}{Y^2(n)}--\frac{c}{Y(n)}+o\left(\frac{1}{n}\right)\right)$$
$$=\frac{n+1}{n}\left(1+\frac{cg^{\prime}(0)}{Y(n)}+\frac{cg^{\prime\prime}(0)}{2n}-\frac{c}{Y(n)}+o\left(\frac{1}{n}\right)\right)$$
$$=\frac{n+1}{n}\left(1+\frac{cg^{\prime\prime}(0)}{2n}+o\left(\frac{1}{n}\right)\right)$$
$$=1+\frac{1+cg^{\prime\prime}(0)/2}{n}+o\left(\frac{1}{n}\right).$$
As result, if $c<\frac{2}{-g^{\prime\prime}(0)}$, then $\E\left(\frac{U_c(n+1)}{U_c(n)}|\F_n\right)<1$ for large enough $n$. 
Hence, under the condition $c<\frac{2}{-g^{\prime\prime}(0)}$, $U_c(n)$ is a supermartingale, and by the martingale convergence theorem, $\sup_{n>0}U_c(n)<R_{c}$ almost surely for some random variable $R_{c}$. Therefore, 
$$Y(n)\geq\frac{cn}{\ln n - ln R_{c}}.$$
As a result, we get
$$\liminf_{n\to\infty} \frac{M(n)\ln n}{n}\geq c$$
almost surely. Since we could take any $c<\frac{2}{-g^{\prime\prime}(0)}$, we get
$$\liminf_{n\to\infty} \frac{M(n)\ln n}{n}\geq \frac{2}{g^{\prime\prime}(0)}.$$
To get the matching upper bound consider
$$Q_c(n):=\frac{e^{\frac{cn}{Y(n+1)}}}{n}.$$
Similar to the $Q_c(n)$ (only sign in exponent and fraction in front of it have changed), we get
$$\E\left(\frac{Q_c(n+1)}{Q_c(n)}|\F_n\right)=\frac{n}{n+1}\E\left(e^{c\left(\frac{n+1}{Y(n+1)}-\frac{n}{Y(n)}\right)}|\F_n\right)$$
$$=\frac{n}{n+1}\E\left(e^{\frac{cn(Y(n+1)-Y(n))}{Y(n)Y(n+1)}-\frac{c}{Y(n+1)}}|\F_n\right)$$
$$=1+\frac{-1-cg^{\prime\prime}(0)/2}{n}+o\left(\frac{1}{n}\right).$$
As a result, if $c>\frac{2}{-g^{\prime\prime}(0)}$, then $\E\left(\frac{U_c(n+1)}{U_c(n)}|\F_n\right)<1$ for large enough $n$.
Hence,
$$\liminf_{n\to\infty} \frac{M(n)\ln n}{n}\leq c$$
for all $c>\frac{2}{-g^{\prime\prime}(0)}$.
Therefore, 
$$\liminf_{n\to\infty} \frac{M(n)\ln n}{n}\leq \frac{2}{g^{\prime\prime}(0)},$$
which concludes the proof of the lemma.
\end{proof}

\section{Proof of the main result}

Let us consider the evolution of $M_1(n)$ (for other types the argument is the same). On step $n+1$, it could be increased in two ways.

First, we could draw an edge from $v_{n+1}$ to the vertex of type $1$ with the highest degree. The probability (conditioned on graph $G_n$) to do so (for each of $m$ possible edges) is at least (exactly, if there is a single vertex with the highest degree) $\frac{M_1(n)+\beta}{D_n}$.

Second, we could draw edges (the same procedure independently repeated $k$ times) to it during an edge step (or for it to be the initial vertex, which happens with probability $1/n$). To do so first we need the initial vertex of a pair to be a type $1$ vertex, which happens with conditional probability $\frac{N_1(n)}{n}$. Then, we need a vertex with the highest degree to appear in the sample. The probability of getting the vertex to the exact position in the sample is at least (exactly if there is a single vertex with the highest degree) $\frac{M_1(n)+\beta}{D_{1}(n)}$, so the probability for the vertex of the maximal degree to be in the sample is at least
$1-\left(1-\frac{M_1(n)+\beta}{D^{1}_n}\right)^{d}.$
As a result, we get an estimate
\begin{equation}
\label{eq:max_degree}
\E (M_1(n+1)- M_1(n)|\F_n)\geq m\frac{M_1(n)+\beta}{D_n} +k\frac{1}{n}+  k\frac{N_1(n)}{n}\left(1-\left(1-\frac{M_1(n)+\beta}{D^{1}_n}\right)^{d}\right).
\end{equation}

Hence, we have the following representation
$$\begin{gathered}
\E (M_1(n+1)- M_1(n)|\F_n)\geq m \frac{M_1(n)}{n}\left(\frac{1}{\frac{D(n)}{n}}\right) \\
 + k\frac{N_1(n)}{n} \left(1-\left(1-\frac{M_1(n)}{n}\frac{1}{\frac{D_{1}(n)}{n}}\right)^{d}\right) + O\left(\frac{1}{n}\right),
\end{gathered}$$
where inequality could be replaced with equality if $M_1(n)$ is achieved on a single vertex.
As a result, for $n\geq n_0$ we get
$$\begin{gathered}\1\{\A^1_{\epsilon}(n_0)\}\E (M_1(n+1)- M_1(n)|\F_n)\geq \1\{\A^1_{\epsilon}(n_0)\}\left(\frac{M_1(n)}{n}\left(\frac{m}{2m+2k+\beta+\epsilon}\right) \right.\\
\left. + (p_1+\epsilon) k\left(1-\left(1-\frac{M_1(n)}{n}\frac{1}{p_1(2 m+2k+\beta+\epsilon)}\right)^{d}\right) + O\left(\frac{1}{n}\right)\right)\\
=\1\{\A^1_{\epsilon}(n_0)\}\left( f_{\epsilon,1}\left(\frac{M_1(n)}{n}\right) + O\left(\frac{1}{n}\right)\right),
\end{gathered}$$
where $f_{\epsilon,1}(x)=\frac{m}{2m+2k+\beta+\epsilon}x+(p_1+\epsilon) k(1-(1-\frac{1}{p_1(2m+2k+\beta+\epsilon)}x)^d)$.
The function $f_{\epsilon,1}(x)$ is concave, $f_{\epsilon,1}(0)=0$, $f_{\epsilon,1}(x)$ converges uniformly to $p_1f\left(\frac{x}{p_1}\right)$ and $$f^{\prime}_{\epsilon,1}(0)=\frac{m +d(1+\frac{\epsilon}{p_1})k}{2m+2k+\beta+\epsilon}\to \frac{m+d}{2m+2k+\beta}.$$ 
Also, $M_1(n+1)- M_1(n)<1+2m$. Therefore, due to Lemma~\ref{lem:approx} we get lower bounds of the Theorem~\ref{th:main}.
Note that the inequality in equation \eqref{eq:max_degree} is due to the possibility of having  multiple vertices of the type $1$ with the highest degree. Hence, upper bounds of the Theorem~\ref{th:main} follows from the following lemma.

\begin{lemma}
\label{lem:hub}
For each $i$ there is a random variable $R_i$, such that for all $n\geq i$ the maximum degree $M_i(n)$ is achieved on the same unique vertex.
\end{lemma}
\begin{proof}
Without loss of generality put $i=1$.
 
We will consider pairs $(M_1(n), \deg_{G_n}(v_i))$ where $v_i$, $i\leq n$, is a given vertex of type $1$, to show that the highest degree is achieved on a single vertex for large enough $n$. To do so, we will use a persistent hub argument, similar to the one used in \cite{Gal16}.
In \cite{Gal16} it was proven that if we consider a two-dimensional random walk $(A_n,B_n)_{n\geq n_0}$ on $\N^2$, which could only move right or up, with 
\begin{equation}
\label{eq:urn_domination}
\Pr\left(A_{n+1}-A_n=1,B_{n+1}-B_n=0|A_n,B_n\right)\geq\frac{A_n+\beta}{A_n+B_n+2\beta},
\end{equation}
$$
\Pr\!\left(A_{n+1}\!\!-\!\!A_n\!=\!1,\!B_{n+1}\!\!-\!\!B_n\!=\!0|A_n,\!B_n\right)
\!+\!\Pr\!\left(A_{n+1}\!\!-\!\!A_n\!=\!0,\!B_{n+1}\!\!-\!\!B_n\!=\!1|A_n,\!B_n\right)\!=\! 1$$
then
\begin{equation}
\label{eq:polyn_est}
    \Pr\left(\exists n:A_n=B_n|A_{n_0}=a,B_{n_0}=1\right)\leq\frac{Q(a)}{2^a}
\end{equation} 
for some polynomial $Q(a)$ (Corollary 8) and the number of moments $n$ with $A_n=B_n$ is finite almost surely for any starting point $(A_{n_0},B_{n_0})$ (Proposition 9).

When applying these results to a pair $(M_1(n), \deg_{G_n}(v_i))$, the estimate \eqref{eq:polyn_est} then implicates that almost surely only a finite number of vertices could become the highest degree vertices among vertices of their type. The second statement would result in only a finite number of changes of leadership, and, hence, after some (random) time, the highest degree is achieved on a single vertex (i.e. the probability that there are no changes of leadership after time $n$ turns to $1$ as $n\to\infty$).

Let $u$ be a vertex of type $1$ that does not have a maximal degree. We check the first inequality from \eqref{eq:urn_domination} by proving such inequality for both steps of adding an edge. In the case we draw an edge from a new vertex to either $u$ or the vertex with the highest degree, probabilities to do so would be exactly $\frac{\deg_{G_n}(u)+\beta}{D(n)}$ for $u$ and at least $\frac{M_1(n)+\beta}{D(n)}$ for a vertex with the highest degree, so the first condition holds. For an edge step, we could increase the degree by choosing a vertex $u$ as the first vertex with probability $\frac{1}{n}$ or by choosing it from the sample. To choose $u$ from the sample, we need at least the vertex with the highest degree to not be in the sample. Let's assume that in this case if $u$ is present in the sample we choose $u$. Such an assumption would increase the probability of choosing $u$. Hence, the conditional average increase of $\deg_{G_n}(u)$ by adding an edge during an edge step would be at most
$$\begin{gathered}
\Pr\left(u=u_{n}^{i}\text{ or } u=w_n^{i}|\F_n\right)\\
\leq\frac{1}{n}+\frac{D_{1}(n)}{D(n)}\left(\left(1-\frac{M_1(n) + \beta}{D_1(n)}\right)^d-\left(1-\frac{M_1(n)+\deg_{G_n}(u) + 2\beta}{D_1(n)}\right)^d\right).
\end{gathered}$$
For the $M_1(n)$, the same increase would be at least
$$\begin{gathered}
\Pr\left(\left\{\deg_{G_n}(u_{n}^{i})=M_1(n)\text{ or } \deg_{G_n}(w_{n}^{i})=M_1(n)\right\}\text{ and } \left\{w_{n}^i \text{ is of type }1\right\}|\F_n\right)\\
\geq\frac{1}{n}+\frac{D_{1}(n)}{D(n)}\left(1-\left(1-\frac{M_1(n) + \beta}{D_1(n)}\right)^d\right).
\end{gathered}$$
Let divide the first increase by the second and prove that it is at most $\frac{\deg_{G_n}(u) + \beta}{M_1(n)+\beta}$, which would result in the existence of the persistent hub. For simplicity denote $B_n:=\deg_{G_n}(u) +\beta$, $C_n:=M_1(n) + \beta$. We get
$$
\begin{gathered}
\frac{\frac{1}{n}+\frac{B_n}{D_1(n)}\sum_{i=0}^{d-1}\left(1-\frac{C_n}{D_1(n)}\right)^i\left(1-\frac{C_n+B_n}{D_1(n)}\right)^{d-i-1}}{\frac{1}{n}
+\frac{C_n}{D_1(n)}\sum_{i=0}^{d-1}\left(1-\frac{C_n}{D_1(n)}\right)^i}\\
=\frac{\frac{B_n}{D_1(n)}\sum_{i=0}^{d-1}\left(1-\frac{C_n}{D_1(n)}\right)^i}{\frac{1}{n}+\frac{C_n}{D_1(n)}\sum_{i=0}^{d-1}\left(1-\frac{C_n}{D_1(n)}\right)^i}\\
+\frac{\frac{1}{n}-\frac{B_n}{D_1(n)}\sum_{i=0}^{d-1}\left(1-\frac{C_n}{D_1(n)}\right)^i
\left(1-\left(1-\frac{C_n+B_n}{D_1(n)}\right)^{d-i-1}\right)}{\frac{1}{n}+\frac{C_n}{D_1(n)}
\sum_{i=0}^{d-1}\left(1-\frac{C_n}{D_1(n)}\right)^i}\\
\leq\frac{B_n}{C_n}
+\frac{\frac{1}{n}-\frac{B_n}{D_1(n)}\frac{C_n+B_n}{D_1(n)} \sum_{i=0}^{d-2}\left(1-\frac{C_n}{D_1(n)}\right)^i \sum_{j=0}^{d-i-2}\left(1-\frac{C_n+B_n}{D_1(n)}\right)^{j}}{\frac{1}{n}+\frac{C_n}{D_1(n)}\sum_{i=0}^{d-1} \left(1-\frac{C_n}{D_1(n)}\right)^i}.
\end{gathered}
$$
Note that $\frac{C_n+B_n}{D_1(n)}$ is separated from $1$. Hence there is a constant $c>0$, such that almost surely
$$\frac{1}{D_1(n)}\frac{C_n\!+\!B_n}{D_1(n)}
\sum_{i=0}^{d-2}\left(1-\frac{C_n}{D_1(n)}\right)^i
\sum_{j=0}^{d-i-2}\!\left(1-\frac{C_n\!+\!B_n}{D_1(n)}\right)^{j}\geq \frac{c}{n}.$$
Therefore, if $\deg_{G_n}(u) + \beta >\frac{1}{c}$, the last term is negative and we get the needed estimate. The condition $\deg_{G_n}(u) + \beta >\frac{1}{c}$ is not significant in random walk estimates, since it only affects the starting point by a constant (we start from $(A,1/c)$ instead of $(A,1)$).

Hence, the number of vertices that could achieve a maximum degree and the number of changes of degree leadership between them is almost surely finite. Therefore, there exists a random variable $N_1$, such that $M_1(n)$ is achieved on the same single vertex for all $n>N_1$. As result, we could replace inequality in \eqref{eq:max_degree} by equality for $n>N_1$ which would not affect convergence (we could replace $\A_{\epsilon}(n_0)$ with $\A_{\epsilon}(n_0)\cap\{n>N_1\}$ in the argument).
\end{proof}

\section*{Acknowledgements.}
The presented work was funded by a grant from the Russian Science Foundation (project No. 24-21-00247).

\end{document}